\documentclass[preprint]{elsarticle}
\usepackage{geometry}
\usepackage{graphicx}
\usepackage{amsmath}
\usepackage{amsthm}
\usepackage{amsfonts}
\usepackage{amssymb, amscd}

\usepackage[latin1]{inputenc}

\newtheorem{thm}{Theorem}[section]
\newtheorem{cor}[thm]{Corollary}
\newtheorem{lem}[thm]{Lemma}
\newtheorem{prop}[thm]{Proposition}
\newtheorem{example}[thm]{Example}
\theoremstyle{definition}
\newtheorem{defn}[thm]{Definition}
\theoremstyle{remark}
\newtheorem{rem}[thm]{Remark}
\numberwithin{equation}{section}

\newcommand{\A}{\mathcal{A}}

\newcommand{\Z}{\mathbb{Z}}

\newcommand{\K}{\mathbb{K}}

\newcommand{\id}{\mbox{id}}

\begin{document}

\title{ Hom-Lie Superalgebras and Hom-Lie
admissible Superalgebras}

\author{Faouzi AMMAR}
\address{ Universit{\'{e}} de Sfax, Facult{\'{e}} des Sciences, B.P.
1171, 3000 Sfax, Tunisia} \ead{Faouzi.Ammar@rnn.fss.tn}

\author{Abdenacer MAKHLOUF \corref{cor1}}

\address{Universit\'{e} de Haute Alsace, LMIA, 4, rue des Fr\`{e}res
Lumi\`{e}re, F-68093 Mulhouse, France}
\ead{Abdenacer.Makhlouf@uha.fr}
 \cortext[cor1]{Corresponding author}



\begin{abstract}
The purpose of this paper is to study Hom-Lie superalgebras, that is
a superspace with a bracket for which the superJacobi identity is twisted by
a homomorphism. This class is a particular case of $\Gamma$-graded quasi-Lie
 algebras introduced by Larsson and Silvestrov. In this paper,
  we characterize  Hom-Lie admissible superalgebras and
  provide
 a construction theorem from which  we derive a one parameter family of Hom-Lie
 superalgebras deforming the orthosymplectic Lie superalgebra. Also,
 we prove a
 $\mathbb{Z}_2$-graded version of a Hartwig-Larsson-Silvestrov Theorem which leads us to a
  construction of
 a $q$-deformed Witt superalgebra.

\end{abstract}


\maketitle


\section*{Introduction}

 The motivations  to study Hom-Lie structures are related to physics and to deformations of  Lie algebras,
 in particular Lie algebras of vector fields. The paradigmatic examples are
 $q$-deformations of Witt and Virasoro
algebras constructed  in pioneering works
 (see \cite{ChaiKuLukPopPresn,ChaiIsKuLuk,ChaiPopPres,Liu,Hu}).

  Larsson and Silvestrov introduced, in \cite{LSgraded}, the class of
 $\Gamma$-graded  quasi-Lie algebras which incorporates as special case
 $\Gamma$-graded  hom-Lie algebras and $\Gamma$-graded  quasi-hom-Lie
 algebras (qhl-algebras). As a particular case, it contains
 Lie superalgebras, Lie algebras as well as Hom-Lie superalgebras and
  Hom-Lie algebras.
 The
main feature of quasi-Lie algebras, quasi-hom-Lie algebras and
Hom-Lie algebras is that the skew-symmetry and the Jacobi identity
are twisted by several deforming twisting maps and also in quasi-Lie
and quasi-hom-Lie algebras the Jacobi identity in general contains 6
twisted triple bracket terms. On the other hand these algebras
correspond to new single and multi-parameter families of algebras
obtained using twisted derivations and constituting deformations and
quasi-deformations of universal enveloping algebras of Lie and color
Lie algebras and of algebras of vector-fields.

The Hom-Lie algebras  were discussed  intensively in  \cite{HLS,LS1,LS2,LS3} while
the graded case was mentioned in \cite{LSgraded}. Hom-associative algebras
were
introduced in \cite{MS},  where it is shown that
 the commutator bracket of a Hom-associative algebra
gives rise to a Hom-Lie algebra and where a classification of Hom-Lie admissible
algebras is established.
 Given a Hom-Lie algebra, there is a universal enveloping Hom-associative
 algebra (see \cite{Yau:EnvLieAlg}).
Dualizing  Hom-associative algebras, one can define Hom-coassociative
coalgebras, Hom-bialgebras and Hom-Hopf algebras (see \cite{MS2,MS4}). It is
shown in \cite{Yau:comodule} that the universal enveloping Hom-associative
algebra carries a structure of
Hom-bialgebra. See also \cite{AM2008,Canepl2009,FregierGohr1,MS3,Yau:homology,Yau:YangBaxter,
Yau:YangBaxter2}
for other works on twisted algebraic structures.

This paper focusses on $\Z _2$-graded Hom-algebras. Mainly, we
introduce and characterize  Hom-Lie admissible superalgebras   and
prove a
 $\mathbb{Z}_2$-graded version of a Hartwig-Larsson-Silvestrov Theorem which leads us to
construct
 a $q$-deformed Witt superalgebra. Also, we
 provide
 a construction theorem from which  we derive a one parameter family of Hom-Lie
 superalgebras deforming the orthosymplectic Lie superalgebra $osp(1,2)$.

In Section 1 of this paper, we summarize the main Hom-algebra
structures and recall the framework of quasi-hom-Lie algebras. In
Section 2, we introduce Hom-associative superalgebras and Hom-Lie
superalgebras. We  provide a different way for constructing Hom-Lie
superalgebras by extending the fundamental construction of Lie
superalgebras from associative superalgebras via supercommutator
bracket. We show that the supercommutator bracket defined using the
multiplication in a Hom-associative superalgebra leads naturally to
a Hom-Lie superalgebra. We also extend a Yau's theorem
\cite{Yau:homology} on Hom-Lie algebras to
 Hom-Lie superalgebras. We show that starting with an ordinary Lie superalgebra and
a superalgebra endomorphism, we may construct a Hom-Lie superalgebra. Moreover,
 we construct a one parameter family of Hom-Lie superalgebras deforming
 the orthosymplectic superalgebra $osp(1,2).$ In Section  3,
we introduce  Hom-Lie admissible superalgebras
and more general $G$-Hom-associative superalgebras, where $G$ is a subgroup
of the permutation group $\mathcal{S}_3$.  We  show that Hom-Lie admissible
superalgebras  are $G$-Hom-associative
superalgebras. As a corollary, we obtain a classification of Lie admissible
 superalgebras.
 These results generalize the classifications obtained in the ungraded case in \cite{MS}, and in
 classical Lie case in \cite{GR04}.
  Section 4 is devoted to  $\mathbb{Z}_2$-graded version of a
  Hartwig-Larsson-Silvestrov Theorem (see \cite{HLS}, Theorem 5). A more general graded
  version of this theorem was mentioned in \cite{LSgraded}. We aim to state it
  explicitly  and
  prove it in the superalgebra case. This leads us, in the last Section, to
    construct
  $q$-deformed Witt superalgebras.
We should point out that all these results could naturally be
extended  to
 color Lie algebras.

\section{Hom-algebras and graded quasi-hom-Lie algebras}
Throughout this paper $\K$ denotes a field of
characteristic $0$.
 We summarize in the following the ungraded definitions of Hom-associative,
 Hom-Leibniz and Hom-Lie algebras (see \cite{MS}). Also, we recall
 the definition of quasi-Lie algebras (see \cite{LS2}) which border
 Hom-Lie algebras, Hom-Lie superalgebras and color Lie algebras, as
 well as quasi-hom-Lie algebras appearing naturally in the study of
 $\sigma$-derivations.
\begin{defn}
A Hom-associative algebra is a triple $(V, \mu, \alpha)$ consisting
of a linear space $V$ over $\K$, a bilinear map $\mu: V\times V
\rightarrow V$ and a  homomorphism $\alpha: V \rightarrow V$
satisfying
$$
\mu(\alpha(x),\mu (y,z))=\mu (\mu (x,y),\alpha (z))
$$
\end{defn}

A class of quasi-Leibniz algebras was introduced
in \cite{LS2} in connection to general quasi-Lie
algebras following the standard Loday's
conventions for Leibniz algebras (i.e. right
Loday algebras). In the context of the subclass
of Hom-Lie algebras one gets a class of
Hom-Leibniz algebras.

\begin{defn}
A Hom-Leibniz algebra is a triple $(V, [\cdot, \cdot], \alpha)$
consisting of a linear space $V$, bilinear map $[\cdot, \cdot]:
V\times V \rightarrow V$ and a homomorphism $\alpha: V \rightarrow
V$  satisfying
\begin{equation} \label{Leibnizalgident}
 [[x,y],\alpha(z)]=[[x,z],\alpha (y)]+[\alpha(x),[y,z]]
\end{equation}
\end{defn}

In terms of the (right) adjoint homomorphisms $Ad_Y: V\rightarrow V$
defined by $Ad_{Y}(X)=[X,Y]$, the identity \eqref{Leibnizalgident}
can be written as
\begin{equation} \label{LeibnizalgidentAdDeriv}
Ad_{\alpha(z)}([x,y]) = [Ad_{z}(x),\alpha(y)] +[\alpha(x),Ad_{z}(y)]
\end{equation}
or in pure operator form
\begin{equation} \label{LeibnizalgidentAdOper}
Ad_{\alpha(z)} \circ Ad_{y} = Ad_{\alpha(y)} \circ Ad_{z} +
Ad_{Ad_{z}(y)} \circ \alpha
\end{equation}

The Hom-Lie algebras were initially introduced by Hartwig, Larson
and Silvestrov in \cite{HLS} motivated initially by examples of
deformed Lie algebras coming from twisted discretizations of vector
fields.

\begin{defn}
A Hom-Lie algebra is a triple $(V, [\cdot, \cdot], \alpha)$
consisting of
 a linear space $V$, bilinear map $[\cdot, \cdot]: V\times V \rightarrow V$ and
 a linear space homomorphism $\alpha: V \rightarrow V$
 satisfying
\begin{eqnarray} \label{skewsymHomLie}
 [x,y]=-[y,x] \quad && {\text{(skew-symmetry)}}
\\{}
\label{JacobyHomLie}
 \circlearrowleft_{x,y,z}{[\alpha(x),[y,z]]}=0
\quad && {\text{(Hom-Jacobi identity)}}
\end{eqnarray}
for all $x, y, z$ from $V$, where
$\circlearrowleft_{x,y,z}$ denotes summation over
the cyclic permutation on $x,y,z$.
\end{defn}
Using the skew-symmetry, one may write the
Hom-Jacobi identity in the form
(\ref{LeibnizalgidentAdDeriv}). Thus, if a
Hom-Leibniz algebra is skewsymmetric, then it is
a Hom-Lie algebra.

We recall in the following the definitions of quasi-Lie and quasi-hom-Lie algebras.

\begin{defn}[\cite{LSgraded}] \label{def:quasiLiealg}
A \emph{quasi-Lie algebra} is a tuple
$(V,[\cdot,\cdot]_V,\alpha,\beta,\omega,\theta)$ where
\begin{itemize}
    \item $V$ is a linear space over $\mathbb{K}$,
    \item $[ \cdot,\cdot]_V:V\times V\to V$ is a bilinear
      map that is called a product or bracket in $V$;
    \item $\alpha,\beta:V\to V$ are linear maps,
    \item $\omega:D_\omega\to \mathfrak{L}_{\mathbb{K}}(V)$ and $\theta:D_\theta\to \mathfrak{L}_{\mathbb{K}}(V)$
      are maps with domains of definition $D_\omega, D_\theta\subseteq V\times
    V$, and where $\mathfrak{L}_{\mathbb{K}}(V)$ denotes the set of
    linear maps on $V$ over $\K$,
\end{itemize}
such that the following conditions hold:
\begin{itemize}
      \item ($\omega$-symmetry) The product satisfies a generalized skew-symmetry condition
        $$[ x,y]_V=\omega(x,y)[ y,x]_V,
        \quad\text{ for all } (x,y)\in D_\omega ;$$
\item (quasi-Jacobi identity) The bracket satisfies a generalized Jacobi identity
    $$\circlearrowleft_{x,y,z}\big\{\,\theta(z,x)\big([\alpha(x),[ y,z]_V]_V+
    \beta [ x,[ y,z]_V ]_V\big)\big\}=0,$$
     for all $(z,x),(x,y),(y,z)\in D_\theta$.
\end{itemize}
\end{defn}
Note that $(\omega(x,y)\omega(y,x)-id)[ x,y]_V=0,$ if $(x,y), (y,x)
\in D_\omega$, which follows from the computation $[
x,y]_V=\omega(x,y)[ y,x]_V =\omega(x,y)\omega(y,x)[ x,y]_V.$

The class of Quasi-Lie algebras incorporates as special cases
\emph{Hom-Lie algebras} and more general \emph{quasi-hom-Lie
algebras (qhl-algebras)} which appear naturally in the algebraic
study of $\sigma$-derivations (see \cite{HLS}) and related
deformations of infinite-dimensional and finite-dimensional Lie
algebras. To get the class of qhl-algebras one specifies
$\theta=\omega$ and restricts attention to maps $\alpha$ and $\beta$
satisfying the twisting condition $[
\alpha(x),\alpha(y)]_V=\beta\circ\alpha [ x,y]_V$. Specifying this
further by taking $D_\omega =V \times V$, $\beta=id$ and
$\omega=-id$, one gets the class of Hom-Lie algebras including Lie
algebras when $\alpha = id$. The class of Hom-Lie superalgebras is
obtained by specifying $\beta=0$, $\omega(x,y)=(-1)^{|x||y|}$ and
 $\theta(x,y)=(-1)^{|x||y|}$, where $x,y$ are elements of $\mathbb{Z}_2$-graded
 space $V$ and $|x|,|y|$ their parities.

\section{Hom-associative Superalgebras and Hom-Lie Superalgebras}
In this Section, we focuss on Hom-associative and Hom-Lie
superalgebras and  provide  a different way for constructing Hom-Lie
superalgebras by extending the fundamental construction of Lie
superalgebras from associative superalgebras via supercommutator
bracket.

Now, let $V$ be a linear superspace over $\K$ that is a
$\mathbb{Z}_2$-graded linear space with a direct sum $V=V_0 \oplus
V_1$. The element of $V_j$, $j=\{0,1\}$, are said to be homogenous
and of parity $j$. The parity of a homogeneous element $x$ is
denoted by $| x |$.

\begin{defn}
A Hom-associative superalgebra is a triple $(V, \mu, \alpha)$
consisting of a superspace $V$, an even bilinear map $\mu: V\times V
\rightarrow V$ and an even  homomorphism $\alpha: V \rightarrow V$
satisfying
$$
\mu(\alpha(x),\mu (y,z))=\mu (\mu (x,y),\alpha (z))
$$
\end{defn}

\begin{defn}
A Hom-Lie superalgebra is a triple $(V, [\cdot, \cdot], \alpha)$
consisting of
 a superspace $V$, an even bilinear map $[\cdot, \cdot]: V\times V \rightarrow V$ and
 an even superspace homomorphism $\alpha: V \rightarrow V$
 satisfying
\begin{eqnarray} \label{skewsymHomLie}
 &[x,y]=-(-1)^{\mid x\mid\mid y\mid}[y,x]\\
\label{JacobyHomsuperLie}
 &(-1)^{\mid x\mid\mid z\mid}[\alpha(x),[y,z]]+(-1)^{\mid z\mid\mid y\mid}
 [\alpha(z),[x,y]]+(-1)^{\mid y \mid\mid x\mid}[\alpha(y),[z,x]]=0
\end{eqnarray}
for all homogeneous element $x, y, z$ in $V$.
\end{defn}

Let $\left( V, [\cdot, \cdot], ,\alpha \right) $ and $\left(
V^{\prime }, [\cdot, \cdot] ^{\prime },\alpha^{\prime }\right) $ be
two Hom-Lie superalgebras. An even homomorphism  $f\ :V\rightarrow
V^{\prime }$ is said to be a \emph{morphism of Hom-Lie
superalgebras} if
\begin{eqnarray} [f(x), f(y)] ^{\prime }&=&f ([x, y])\quad \forall x,y\in V
\\ f\circ \alpha&=&\alpha^{\prime
}\circ f
\end{eqnarray}
Morphisms of Hom-associative superalgebras are defined similarly.
\begin{rem}
We recover the classical Lie superalgebra and associative
superalgebra when $\alpha =\id$. The Hom-Lie algebra and
Hom-associative algebras are obtained when the part of parity one is
trivial.
\end{rem}
\begin{example}[2-dimensional abelian Hom-Lie  superalgebra]

Every  bilinear map $\mu$ on a $2$- dimensional linear superspace
$V=V_0 \oplus V_1$, where $V_0$ is generated by $x$ and $V_1$ is
generated by $y$ and  such that $[x,y]=0$ defines a Hom-Lie
superalgebra for any homomorphism $\alpha$ of superalgebra. Indeed,
the graded Hom-Jacobi identity is satisfied for any triple
$(x,x,y)$.
\end{example}

\begin{example}[Affine Hom-Lie  superalgebra ]
Let $V=V_0 \oplus V_1$ be a $3$-dimensional superspace where $V_0$
is generated by $e_1 , e_2$ and $V_1$ is generated by $e_3$. The
triple
 $(V,[\cdot,\cdot],\alpha)$ is a Hom-Lie superalgebra defined by $[e_1
 ,e_2]=e_1,[e_1,e_3]=[e_2,e_3]=[e_3,e_3]=0$ and $\alpha$ is any
 homomorphism.
\end{example}
In the following, we show that the supercommutator bracket defined
using the multiplication in a Hom-associative algebra leads
naturally to Hom-Lie superalgebra.
\begin{prop}\label{Supercommutator}
Given any Hom-associative superalgebra $(V,\mu,\alpha )$ one can define the supercommutator on homogeneous elements by
   $$ [x,y] = \mu (x,y) - ( - 1)^{ | x | | y |} \mu (y,x)$$
and then extending by linearity to all elements. Then  $(V,[\cdot,\cdot],\alpha )$  is a Hom-Lie superalgebra.
\end{prop}
\begin{proof}
The bracket is obviously supersymmetric and with a direct
computation we have
$$
\begin{array}{c}
  (-1)^{|x||z|}[\alpha (x),[y,z]]+
  (-1)^{|z||y|}[\alpha (z),[x,y]]+
  (-1)^{|y||x|}[\alpha(y),[z,x]]= \\
  (-1)^{|x||z|}\mu(\alpha (x),\mu(y,z))-(-1)^{|x||z|+|y||z|}\mu(\alpha(x),\mu(z,y))
  \\-(-1)^{|x||y|}\mu(\mu(y,z),\alpha(x))+(-1)^{|x||y|+|y||z|}
\mu(\mu(z,y),\alpha(x))\\
+(-1)^{|y||x|}\mu(\alpha
(y),\mu(z,x))-(-1)^{|x||y|+|z||x|}\mu(\alpha(y),\mu(x,z))
 \\ -(-1)^{|y||z|}\mu(\mu(z,x),\alpha(y))+(-1)^{|y||z|+|z||x|}
\mu(\mu(x,z),\alpha(y))\\
+(-1)^{|z||y|}\mu(\alpha(z),\mu(x,y))-(-1)^{|y||z|+|x||y|}\mu(\alpha(z),\mu(y,x))
  \\-(-1)^{|z||x|}\mu(\mu(x,y),\alpha(z))+(-1)^{|z||x|+|x||y|}
\mu(\mu(y,x),\alpha(z)) =0
\end{array}
$$
\end{proof}

We extend in the following the Yau's theorem (see
\cite{Yau:homology}) into  the graded case. The following theorem
gives a way to construct Hom-Lie superalgebras, starting from a Lie
superalgebra and an even  superalgebra endomorphism.

\begin{thm} \label{thm:SALmorphism}
Let $(V,[\cdot,\cdot])$ be a Lie superalgebra  and  $\alpha : V\rightarrow V$ be
an even  Lie superalgebra endomorphism.
Then
$(V,[\cdot,\cdot]_\alpha,\alpha)$, where $[x,y]_\alpha=\alpha([x,y])$,  is a Hom-Lie
superalgebra.

Moreover, suppose that  $(V',[\cdot,\cdot]')$ is another Lie superalgebra and  $\alpha ' : V'\rightarrow V'$ is a Lie superalgebra endomorphism. If $f:V\rightarrow V'$ is a Lie superalgebra morphism that satisfies $f\circ\alpha=\alpha'\circ f$ then
$$f:(V,[\cdot,\cdot]_\alpha,\alpha)\longrightarrow (V',[\cdot,\cdot]',\alpha ')
$$
is a morphism of Hom-Lie superalgebras.
\end{thm}
\begin{proof}
We show that $(V,[\cdot,\cdot]_\alpha,\alpha)$ satisfies the Hom-superJacobi
  identity \ref{JacobyHomsuperLie}. Indeed
\begin{align*}
\circlearrowleft_{x,y,z}{(-1)^{\mid x \mid\mid z\mid}[\alpha(x),[y,z]_\alpha]_
\alpha}&=\circlearrowleft_{x,y,z}{(-1)^{\mid x\mid\mid z\mid}\alpha ([\alpha(x),\alpha([y,z])])}\\
&=\alpha^2(\circlearrowleft_{x,y,z}{(-1)^{\mid x\mid\mid z\mid} [x,[y,z]])}\\
&=\alpha^2(\circlearrowleft_{x,y,z}{(-1)^{\mid x\mid\mid z\mid} [x,[y,z]])}\\
&=0
\end{align*}

The second assertion follows from
$$ f\circ [\cdot,\cdot]_\alpha = f\circ \alpha \circ [\cdot,\cdot]
= \alpha ' \circ f \circ [\cdot,\cdot] = \alpha ' \circ [\cdot,\cdot]' \circ f
= [\cdot,\cdot]'_{\alpha '} \circ f. $$
\end{proof}

\begin{example}We construct an example of Hom-Lie superalgebra, which is not
a Lie superalgebra starting from the orthosymplectic Lie superalgebra.
We consider in the sequel the matrix realization of this Lie superalgebra.

Let $osp(1,2)=V_0 \oplus V_1$  \ be  the Lie superalgebra where
$V_0$ is generated by:
$$ H=\left(
  \begin{array}{ccc}
  1 & 0& 0 \\
  0 &0 & 0 \\
    0 & 0 & -1\\
  \end{array}\right), \ \ X=\left(
  \begin{array}{ccc}
  0 & 0 & 1 \\
  0 & 0 & 0 \\
  0 & 0 & 0\\
  \end{array}
\right),\ \ Y=\left(
  \begin{array}{ccc}
  0 & 0& 0 \\
  0 & 0 & 0 \\
  1 & 0 & 0\\
  \end{array}
\right), $$ and $V_1$ is generated by:

$$ F=\left(
  \begin{array}{ccc}
  0 & 0 & 0 \\
  1 & 0 & 0 \\
  0 & 1 & 0\\
  \end{array}
\right) , \ \ G=\left(
  \begin{array}{ccc}
  0 & 1& 0 \\
  0 & 0 & -1 \\
  0 & 0 & 0\\
  \end{array}
\right) .$$

The defining relations (we give only the ones with non zero values
in the right hand side) are
$$[H,X]=2X, \ [H,Y]=-2Y,\ [X,Y]=H,$$
$$[Y,G]=F,\ [X,F]=G,\ [H,F]=-F, \ [H,G]=G,$$
$$\ [G,F]=H, \ [G,G]=-2X,\ [F,F]=2Y.
$$
 Let $\lambda \in \mathbb{R}^*$, we consider the linear map
$\alpha_{\lambda}: osp(1,2)\rightarrow osp(1,2)$ defined by:
$$
\alpha_{\lambda}(X)=\lambda^2 X, \ \
\alpha_{\lambda}(Y)=\frac{1}{\lambda^2}Y, \ \
\alpha_{\lambda}(H)=H,\ \alpha_{\lambda}(F)=\frac{1}{\lambda}F, \ \
\alpha_{\lambda}(G)=\lambda G .$$

We provide a family of Hom-Lie superalgebras
$osp(1,2)_\lambda=(osp(1,2),[\cdot,\cdot]_{\alpha_\lambda} ,
\alpha_\lambda)$ where the Hom-Lie superalgebra bracket
$[\cdot,\cdot]_{\alpha_\lambda}$ on the basis elements is given, for
$\lambda\neq 0$, by:
$$[H,X]_{\alpha_\lambda}=2\lambda^2 X,\ \
[H,Y]_{\alpha_\lambda}=-\frac{2}{\lambda^2}Y, \ \
[X,Y]_{\alpha_\lambda}=H,$$ $$
[Y,G]_{\alpha_\lambda}=\frac{1}{\lambda}F,\ \
[X,F]_{\alpha_\lambda}=\lambda G,\ \
[H,F]_{\alpha_\lambda}=-\frac{1}{\lambda}F,\ \
[H,G]_{\alpha_\lambda}=\lambda G,$$ $$ [G,F]_{\alpha_\lambda}=H,\ \
[G,G]_{\alpha_\lambda}=-2\lambda^2X,\ \
[F,F]_{\alpha_\lambda}= \frac{2}{\lambda^2}Y. $$

These Hom-Lie superalgebras are not Lie superalgebras for $\lambda\neq 1$.

Indeed, the left hand side of the superJacobi identity
(\ref{JacobyHomsuperLie}), for $\alpha =\id$,  leads to
$$[X,[Y,H]-[H,[X,Y]]+[Y,[H,X]]=\frac{2(1-\lambda ^4)}{\lambda^{2}}Y,
$$
and also
$$[H,[F,F]-[F,[H,F]]+[F,[F,H]]=\frac{4(\lambda -1)}{\lambda^{4}}Y.
$$
Then, they do not vanish for $\lambda\neq 1.$

\end{example}

\section{Hom-Lie-Admissible Superalgebras}
We introduce and discuss in this Section the Hom-Lie admissible
superalgebras. The Lie admissible algebras were introduced by A. A.
Albert in 1948.  We extend to graded algebras the concept of
Hom-Lie-Admissible algebras studied in \cite{MS}. This study borders
also an extension to graded case of the Lie-admissible algebras
discussed in \cite{GR04}.

Let $\A =(V, \mu, \alpha)$ be a Hom-superalgebra, that is a
superspace $V$ with an even bilinear map $\mu$ and an even  linear
map $\alpha$ satisfying eventually identities. Let $[x,y]=
\mu(x,y)-(-1)^{|x||y|}\mu(y,x) $, for all homogeneous element $x,y
\in V$, be the associated supercommutator. We call
$\alpha$-associator of a multiplication
  $\mu$
 a trilinear map
$\mathfrak{as}_{\alpha}$ on $V$  defined for  $x_1,x_2,x_3 \in V$  by
$$
\mathfrak{as}_{\alpha}(x_1,x_2,x_3)=\mu(\alpha(x_{1}),\mu (x_{2},x_3))-\mu (\mu (x_1,x_2),\alpha
(x_{3})).
$$
\begin{defn} \label{def:Lieadmis}
Let $\A=(V, \mu, \alpha)$ be a Hom-superalgebra  on $V$ defined by
an even multiplication $\mu$ and an even homomorphism $\alpha$.
Then $\A$ is said to be Hom-Lie admissible superalgebra over $V$ if
the bracket defined for all homogeneous element $x,y \in V$ by
\begin{equation} \label{commutator}
[ x,y ]=\mu (x,y)-(-1)^{|x||y|}\mu (y,x )
\end{equation}
 satisfies the
Hom-superJacobi identity (\ref{JacobyHomsuperLie}).

\end{defn}
\begin{rem}
Since the supercommutator bracket \eqref{commutator}
is always supersymmetric, this makes any Hom-Lie
admissible superalgebra into a Hom-Lie superalgebra.
\end{rem}
\begin{rem}
 As mentioned  in in the proposition (\ref{Supercommutator}), any
 associative superalgebra is a Hom-Lie admissible superalgebra.
\end{rem}
\begin{lem}\label{JacobiSuperCommutator}
Let $\A=(V, \mu, \alpha)$ be a Hom-superalgebra and
$[ -,- ]$ be the associated supercommutator then
\begin{eqnarray}\label{LieAdmiIdentity}
&&\circlearrowleft_{x,y,z}{(-1)^{|x||z|}[\alpha (x),[y,z]]}=\\
&&\quad (-1)^{|x||z|} \mathfrak{as}_{\alpha}(x,y,z) +
(-1)^{|y||x|}\mathfrak{as}_{\alpha}(y,z,x)+(-1)^{|z||y|}
\mathfrak{as}_{\alpha}(z,x,y)\nonumber\\
&&\quad
-(-1)^{|x||z|+|y||z|}\mathfrak{as}_{\alpha}(x,z,y)-(-1)^{|x||y|+|y||z|}
\mathfrak{as}_{\alpha}(z,y,x)-(-1)^{|x||y|+|x||z|}
\mathfrak{as}_{\alpha}(y,x,z)\nonumber
\end{eqnarray}
\end{lem}
\begin{proof}
By straightforward calculation, we have
\begin{eqnarray*}
&&\circlearrowleft_{x,y,z}(-1)^{|x||z|}[\alpha (x),[y,z]]
=\circlearrowleft_{x,y,z}(-1)^{|x||z|}[\alpha (x),\mu (y,z)-(-1)^{|y||z|}\mu (z,y)]
\\
&& =(-1)^{|x||z|}\mu (\alpha (x),\mu (y,z))-(-1)^{|x||z|}\mu (\mu (y,z),\alpha (x))\\
&&\quad +(-1)^{|y||x|}\mu (\alpha (y),\mu (z,x))-(-1)^{|y||z|}\mu (\mu (z,x),\alpha (y))\\
&&\quad +(-1)^{|z||y|}\mu (\alpha (z),\mu (x,y))-(-1)^{|z||y|}\mu (\mu (x,y),\alpha (z))\\
&&\quad -(-1)^{|x||z|+|y||z|}\mu (\alpha (x),\mu (z,y))+
(-1)^{|x||z|+|y||z|}\mu (\mu (z,y),\alpha (x))\\
&&\quad -(-1)^{|y||x|+|z||x|}\mu (\alpha (y),\mu (x,z))+
(-1)^{|z||x|+|y||z|}\mu (\mu (x,z),\alpha (y))\\
&&\quad -(-1)^{|z||y|+|x||y|}\mu (\alpha (z),\mu (y,x))+
(-1)^{|x||y|+|z||x|}\mu (\mu (y,x),\alpha (z))
\\
&&=
(-1)^{|x||z|}
\mathfrak{as}_{\alpha}(x,y,z)+
(-1)^{|y||x|}\mathfrak{as}_{\alpha}(y,z,x)+(-1)^{|z||y|}
\mathfrak{as}_{\alpha}(z,x,y)\\
&&\quad -(-1)^{|x||z|+|y||z|}\mathfrak{as}_{\alpha}(x,z,y)-(-1)^{|x||y|+|y||z|}
\mathfrak{as}_{\alpha}(z,y,x)-(-1)^{|x||y|+|x||z|}
\mathfrak{as}_{\alpha}(y,x,z)
\end{eqnarray*}
\end{proof}
\begin{rem}
If $\alpha =\id$, then we obtain a formula expressing the left hand side of the
classical superJacobi identity in terms of classical associator.
\end{rem}
 In the following we aim to characterize the Hom-Lie admissible superalgebras
 in terms of $\alpha$-associator.
 We introduce the notation
$$S(x,y,z):= (-1)^{|x||z|}\mathfrak{as}_{\alpha}(x,y,z)+
(-1)^{|y||x|}\mathfrak{as}_{\alpha}(y,z,x)+(-1)^{|z||y|}
\mathfrak{as}_{\alpha}(z,x,y).
$$
Then, we have the following properties.

\begin{prop}
Let  $\A =(V, \mu, \alpha)$ be a Hom-superalgebra, then $\A$ is a
Hom-Lie admissible superalgebra if and
only if it satisfies
\begin{equation}\label{Srel} S(x,y,z)=(-1)^{|x||y|+|x||z|+|y||z|}S(x,z,y)
\end{equation}
for any homogenous elements  $x,y,z \in V.$
\end{prop}
\begin{proof} We have
\begin{eqnarray*}
&& S(x,y,z)-(-1)^{|x||y|+|x||z|+|y||z|}S(x,z,y)
\\&&=(-1)^{|x||z|}
\mathfrak{as}_{\alpha}(x,y,z)+
(-1)^{|y||x|}\mathfrak{as}_{\alpha}(y,z,x)+(-1)^{|z||y|}
\mathfrak{as}_{\alpha}(z,x,y)\\
&& \quad -(-1)^{|x||y|+|x||z|+|y||z|}((-1)^{|x||y|}\mathfrak{as}_{\alpha}(x,z,y)+
(-1)^{|z||x|}\mathfrak{as}_{\alpha}(z,y,x)+(-1)^{|y||z|}
\mathfrak{as}_{\alpha}(y,x,z))
\\
&&=(-1)^{|x||z|}
\mathfrak{as}_{\alpha}(x,y,z)+
(-1)^{|y||x|}\mathfrak{as}_{\alpha}(y,z,x)+(-1)^{|z||y|}
\mathfrak{as}_{\alpha}(z,x,y)\\
&&\quad -(-1)^{|x||z|+|y||z|}\mathfrak{as}_{\alpha}(x,z,y)-(-1)^{|x||y|+|y||z|}
\mathfrak{as}_{\alpha}(z,y,x)-(-1)^{|x||y|+|x||z|}
\mathfrak{as}_{\alpha}(y,x,z)\\
 &&= \circlearrowleft_{x,y,z}{(-1)^{|x||z|}[x,[y,z]]} \quad
 \text{(lemma \ref{JacobiSuperCommutator})}
\end{eqnarray*}
Then the Hom-superJacobi identity (\ref{JacobyHomsuperLie}) is satisfied if and only if the condition
(\ref{Srel}) holds.
\end{proof}

 In the following, we provide a classification of Hom-Lie admissible superalgebras
 using the symmetric group $\mathcal{S}_3$. We extend to graded case the
  notion of $G$-Hom-associative algebras
 which was introduced in the classical ungraded Lie case in (\cite{GR04}) and
  developed for the Hom-Lie case in (\cite{MS}).

Let $\mathcal{S}_3$ be the permutation group generated by $\sigma_1
,\sigma_2.$ We extend a permutation $\tau \in \mathcal{S}_3$ to a
map $\tau : V^{\times 3} \rightarrow V^{\times 3}$ defined for
$x_1,x_2,x_3\in V$ by
$\tau(x_1,x_2,x_3)=(x_{\tau(1)},x_{\tau(2)},x_{\tau(3)}).$ We keep
for simplicity
 the same notation. In particular,
 $\sigma_1(x_1,x_2,x_3)=(x_{2},x_{1},x_{3})$ and
 $\sigma_2(x_1,x_2,x_3)=(x_{1},x_{3},x_{2}).$

 We introduce a notion of a parity of transposition $\sigma_i$ where $i\in \{1,2\}$, by setting
   $$|\sigma_i(x_1,x_2,x_3)|=|x_i||x_{i+1}|.$$
 We assume that the parity of the identity is $0$ and for the composition
 $\sigma_i\sigma_j$, it is defined by
\begin{eqnarray*}|\sigma_i\sigma_j(x_1,x_2,x_3)|&=&
 |\sigma_j(x_1,x_2,x_3)|+|\sigma_i(\sigma_j(x_1,x_2,x_3))|\\
 \ &=& |\sigma_j(x_1,x_2,x_3)|+
 |\sigma_i(x_{\sigma _j(1)},x_{\sigma_ j(2)},x_{\sigma_j(3)})|
 \end{eqnarray*}
We define by induction the parity for any composition. For the elements
$\id,\sigma_1 ,\sigma_2,\sigma_1 \sigma_2,\sigma_2 \sigma_1,
\sigma_2 \sigma_1 \sigma_2$ of
$\mathcal{S}_3$, we obtain
\begin{eqnarray*}
\mid \id (x_1,x_2,x_3)\mid&=&0,\\
\mid\sigma_1(x_1,x_2,x_3)\mid &=&|x_1||x_{2}|
 ,\\
 \mid\sigma_2(x_1,x_2,x_3)\mid &=&|x_2||x_{3}|,\\
 \mid \sigma_1 \sigma_2(x_1,x_2,x_3)\mid &=&|x_2||x_{3}|+|x_1||x_{3}|
 ,\\
  \mid \sigma_2 \sigma_1(x_1,x_2,x_3)\mid &=&|x_1||x_{2}|+|x_1||x_{3}|,\\
\mid \sigma_2 \sigma_1 \sigma_2(x_1,x_2,x_3)\mid &=&|x_2||x_{3}|+|x_1||x_{3}|+
|x_1||x_{2}|.
\end{eqnarray*}

Now, we  express the condition of Hom-Lie admissibility of a Hom-superalgebra
using  permutations.

\begin{lem}\label{JacobiSuperCommutator2}
A Hom-superalgebra $\A=(V, \mu, \alpha)$ is a Hom-Lie admissible
superalgebra if the following condition holds
\begin{equation}\label{LieAdmiIdentity2}
\sum_{\tau\in \mathcal{S}_3} {(-1)^{\varepsilon ({\tau})}
(-1)^{|\tau(x_1,x_2,x_3)|}\mathfrak{as}_{\alpha}\circ \tau(x_1,x_2,x_3)} =0
\end{equation}
where $x_i$ are in $V$, $(-1)^{\varepsilon ({\tau})}$ is the
signature of the permutation
$\tau$ and $|\tau(x_1,x_2,x_3)|$ is the parity of $\tau$.
\end{lem}
\begin{proof}
By straightforward calculation, the associated supercommutator bracket
satisfies
\begin{equation*}
\circlearrowleft_{x_1,x_2,x_3}(-1)^{|x_1||x_3|}[\alpha (x_1),[x_2,x_3]]
=(-1)^{|x_1||x_3|}\sum_{\tau\in \mathcal{S}_3} {(-1)^{\varepsilon ({\tau})}
(-1)^{|\tau(x_1,x_2,x_3)|}\mathfrak{as}_{\alpha}\circ \tau(x_1,x_2,x_3)}.
\end{equation*}
\end{proof}
It turns out that, for the associated supercommutator   of a Hom-superalgebra,
  the Hom-superJacobi identity (\ref{JacobyHomsuperLie}) is equivalent to
\begin{equation*}
\sum_{\tau\in \mathcal{S}_3} {(-1)^{\varepsilon ({\tau})}
(-1)^{|\tau(x_1,x_2,x_3)|}\mathfrak{as}_{\alpha}\circ \tau(x_1,x_2,x_3)}=0.
\end{equation*}
We introduce the notion of $G$-Hom-associative superalgebra, where $G$ is
 a subgroup of the permutations group $\mathcal{S}_3$.
\begin{defn}
Let $G$ be a subgroup of the permutations group $\mathcal{S}_3$, a
Hom-superalgebra on $V$ defined by the multiplication $\mu$ and a
homomorphism $\alpha$ is said to  be $G$-Hom-associative superalgebra if
\begin{equation}\label{admi}
\sum_{\tau\in G} {(-1)^{\varepsilon ({\tau})}
(-1)^{|\tau(x_1,x_2,x_3)|}\mathfrak{as}_{\alpha}\circ \tau(x_1,x_2,x_3)}=0.
\end{equation}
where $x_i$ are in $V$, $(-1)^{\varepsilon ({\tau})}$ is the
signature of the permutation
and $|\tau(x_1,x_2,x_3)|$ is the parity of $\tau$ defined above.
\end{defn}

The following result is a graded version  of the results
obtained in (\cite{GR04,MS}).

\begin{prop}
Let $G$ be a subgroup of the permutations group
$\mathcal{S}_3$. Then any $G$-Hom-associative
superalgebra is a Hom-Lie admissible superalgebra.
\end{prop}
\begin{proof}The supersymmetry follows straightaway from the
definition.

We have a subgroup $G$ in $\mathcal{S}_3$. Take the set of conjugacy
class $\{g G\}_{g\in I}$  where $I\subseteq G$, and for any
$\tau_1, \tau_2\in I,\sigma_1 \neq \tau_2 \Rightarrow \tau_1
G\bigcap \tau_2 G =\emptyset$. Then

\begin{eqnarray*}&& \sum_{\tau\in \mathcal{S}_3}{(-1)^{\varepsilon ({\tau})}(-1)^{|\tau(x_1,x_2,x_3)|}
\mathfrak{as}_{\alpha}\circ
\tau}(x_1,x_2,x_3)=\\&& \sum_{\tau_1\in I}{\sum_{\tau_2\in \tau_1
G}{(-1)^{\varepsilon ({\tau_2})}(-1)^{|\tau_2(x_1,x_2,x_3)|}\mathfrak{as}_
{\alpha}\circ \tau_2}}(x_1,x_2,x_3)=0
\end{eqnarray*}
where $(x_1,x_2,x_3)\in V$, with $V$  the underlaying superspace of the $G$-Hom-associative
superalgebra.
\end{proof}
It follows that in particular, we have:
\begin{cor}
Let $G$ be a subgroup of the permutations group
$\mathcal{S}_3$. Then any $G$-associative
superalgebra is a Lie admissible superalgebra.
\end{cor}
Now, we provide a classification of the Hom-Lie admissible
superalgebras through $G$-Hom-associative superalgebras. The
ungraded $G$-associative algebras in classical case was studied in
\cite{GR04}, then extended to $G$-Hom-associative algebras in
(\cite{MS}).

 The subgroups of $\mathcal{S}_3$ are
$$G_1=\{Id\}, ~G_2=\{Id,\sigma_{1}\},~G_3=\{Id,\sigma_{2}\},~G_4=\{Id,\sigma_2\sigma_1\sigma_2\},~G_5=A_3 ,~G_6=\mathcal{S}_3,$$
where $A_3$ is the alternating group.

We obtain the following type of Hom-Lie
admissible superalgebras.
\begin{itemize}
\item The  $G_1$-Hom-associative superalgebras  are the Hom-associative
superalgebras defined above.

\item The  $G_2$-Hom-associative superalgebras satisfy the identity
\begin{equation}\nonumber
\mu(\alpha(x),\mu (y,z))- \mu (\mu (x,y),\alpha
(z))=(-1)^{|x||y|}(\mu(\alpha(y),\mu (x,z))-\mu (\mu (y,x),\alpha
(z)))
\end{equation}
\item The  $G_3$-Hom-associative superalgebras satisfy the identity
\begin{equation}\nonumber
\mu(\alpha(x),\mu (y,z))-\mu (\mu (x,y),\alpha
(z))=(-1)^{|y||z|}(\mu(\alpha(x),\mu (z,y))-\mu (\mu (x,z),\alpha
(y)))
\end{equation}
\item The  $G_4$-Hom-associative superalgebras satisfy the identity
\begin{equation}\nonumber
\mu(\alpha(x),\mu (y,z))-\mu (\mu (x,y),\alpha
(z))=(-1)^{|x||y|+|x||z|+|y||z|}(\mu(\alpha(z),\mu (y,x))-\mu (\mu
(z,y),\alpha (x)))
\end{equation}
\item The  $G_5$-Hom-associative superalgebras satisfy the identity
\begin{eqnarray}\nonumber
\nonumber \mu(\alpha(x),\mu (y,z))-\mu (\mu (x,y),\alpha
(z))+(-1)^{|x||y|+|x||z|}(\mu(\alpha(y),\mu (z,x)+\mu (\mu
(y,z),\alpha
(x)))= \\
\nonumber -(-1)^{|x||z|+|y||z|}(\mu(\alpha(z),\mu (x,y))+\mu (\mu
(z,x),\alpha (y)))
\end{eqnarray}\nonumber
\item The  $G_6$-Hom-associative superalgebras are the Hom-Lie admissible
superalgebras.
\end{itemize}

\begin{rem}
Moreover, if in the previous identities we consider $\alpha=\id $, then we
obtain a classification of Lie-admissible superalgebras.
\end{rem}
\begin{rem}One may call $G_2$-Hom-associative (resp. $G_2$-associative)
 superalgebras  \emph{Hom-Vinberg superalgebras} (resp. \emph{Vinberg superalgebras})
and $G_3$-Hom-associative (resp. $G_3$-associative) superalgebras
\emph{Hom-pre-Lie superalgebras} (resp. \emph{pre-Lie
superalgebras}). Notice that a Hom-pre-Lie superalgebra is the
opposite algebra of a Hom-Vinberg superalgebra. Therefore, they
actually form a same class.
\end{rem}

\section{$\mathbb{Z}_2$-Graded  Hartwig-Larsson-Silvestrov theorem}

In this Section, we describe and prove  the
Hartwig-Larsson-Silvestrov theorem (see \cite{HLS}, theorem 5) in
the case of $\mathbb{Z}_2$-graded algebra. A more general graded
version of this theorem was mentioned in (\cite{LSgraded}). We aim
to consider it deeply for superalgebras case.

  Let $\A=\A_0 \oplus \A_1$ be an associative superalgebra. We assume that $\A$
   is supercommutative, that is for homogeneous elements $a,b$, the identity
    $a b=(-1)^{|a||b|}ba $ holds. For example,  $\A_0 = \mathbb{C}[t,t^{-1}]$ and $\A_1 =\theta
\A_0$ where $\theta$ is the Grassman variable $(\theta^2 = 0)$. Let $\sigma$ be  an even superalgebra
 endomorphism of $\A$. Then,
$\A$ is a bimodule over itself, the left (resp. right) action is
defined by $a \cdot_{l} b= \sigma(a) b$ (resp. $b \cdot_r a=b a$). For
simplicity, we denote the module multiplication by a \emph{dot} and the superalgebra
multiplication by juxtaposition. In the sequel the elements of $\A$ are supposed
to be homogeneous.

\begin{defn}
Let $i\in \{0,1\}$.  A $\sigma$-derivation $D_i$ on $\A$ is an endomorphism satisfying:
 $$D_i (a  b)=D_i(a) b +(-1)^{i|a|}\sigma(a) D_i (b)$$ where $a,b\in
 \A$ are homogeneous element and  $|a|$ is the parity of $a$.

A $\sigma$-derivation $D_0$ is called  even $\sigma-$derivation and $D_1$
is called  odd
$\sigma$-derivation. The set of all $\sigma$-derivations is denoted by
$Der_\sigma(\A)$. Therefore, $Der_\sigma(\A)=Der_\sigma(\A)_0\oplus
Der_\sigma(\A)_1$, where $Der_\sigma(\A)_0$ (resp.  $Der_\sigma(\A)_1$) is
the space of even (resp.
odd) $\sigma$-derivations. The structure of $\A$-supermodule of $Der_\sigma(\A)$ is
as usual. Let $D\in Der_\sigma(\A)$, the annihilator $Ann(D)$   is the
set of all $a\in\A$ such that $a\cdot D =0$. We set $\A \cdot D=\{a\cdot D
: \ a\in \A \}$ be an $\A$-subsupermodule of $Der_\sigma(\A)$.
\end{defn}



Let $\sigma:\A
\rightarrow \A$ be a fixed even endomorphism, $\Delta$ an even $\sigma$-derivation
 ($ \Delta \in Der_\sigma(\A)_0 $)
 and  $\delta$ be an element in $\A$.

\begin{thm}\label{HLS} If \begin{equation}\label{HLSCond1} \sigma(Ann \Delta)\subset Ann \Delta
\end{equation}  holds then
the map $[-,-]_{\sigma}: \A\Delta \times
\A\Delta \rightarrow \A\Delta$ defined by setting
\begin{equation}\label{*}
[a\cdot\Delta,b\cdot\Delta]_{\sigma} = (\sigma(a)\cdot\Delta)\circ (b\cdot \Delta)
 - (-1)^{ \mid a\mid\mid b\mid}(\sigma(b)\cdot\Delta)\circ(a\cdot\Delta)  \ \ for
 \ \ a,b \in \A
 \end{equation}
where $\circ$ denotes  the composition of functions, is a well-defined
 superalgebra bracket on
 the superspace $\A\cdot\Delta$ and satisfies the following
identities for $a,b,c \in \A$
\begin{eqnarray}\label{**}[a\cdot\Delta,b\cdot\Delta]_{\sigma} &=&
 (\sigma(a)\Delta (b) -
(-1)^{ \mid a\mid\mid b\mid}\sigma(b)\Delta(a))\cdot\Delta \\
\label{**2}
[ a\cdot\Delta, b\cdot\Delta ]_{\sigma} &=&
 - (-1)^{ \mid a\mid\mid b\mid}[b\cdot\Delta,a\cdot\Delta]_\sigma
\end{eqnarray}
In addition, if \begin{equation}
\label{HLSCond2}\Delta(\sigma(a))=\delta \sigma(\Delta(a)) \ \ \text {for }
 \ \ a\in \A
 \end{equation}
 holds, then
\begin{eqnarray}\label{***}
\circlearrowleft_{a,b,c}{(-1)^{|a||c|}([\sigma(a)\cdot\Delta,[b\cdot\Delta,c\cdot
\Delta]_{\sigma}]_{\sigma}
+\delta [a\cdot\Delta,[b\cdot\Delta,c\cdot\Delta]_{\sigma}]_{\sigma})}= 0
\end{eqnarray}
\end{thm}
\begin{proof}

We show first  that $[-,-]_{\sigma}$ is a well defined function.
That is if
$a_{1}\cdot  \Delta = a_{2}\cdot  \Delta$ then $[a_{1}\cdot \Delta
,b\cdot \Delta]_{\sigma}=[a_{2} \cdot \Delta ,b \cdot \Delta]_{\sigma}$ and
$[b \cdot\Delta
,a_{1} \cdot\Delta]_{\sigma}=[b\cdot\Delta , a_{2}\cdot
\Delta]_{\sigma} $ for
$a_{1},a_{2},b\in \A.$

We compute
\begin{eqnarray*}
&&[a_{1} \cdot\Delta ,b\cdot \Delta]_{\sigma}
-[a_{2} \cdot \Delta ,b \cdot\Delta]_{\sigma}\\
&&=(\sigma(a_{1})\cdot \Delta)\circ(b\cdot\Delta)
-(-1)^{\mid a_{1}\mid\mid b\mid}(\sigma(b)\cdot \Delta)\circ(a_{1}\cdot\Delta)
\\ && -(\sigma(a_{2})\cdot\Delta)\circ(b \cdot\Delta)
+(-1)^{\mid a_{2}\mid \mid b\mid}(\sigma(b)\cdot\Delta )\circ(a_{2}\cdot\Delta)
\\
&&=(\sigma(a_{1}-a_{2})\cdot\Delta)\circ(b\cdot\Delta)
-(-1)^{\mid a_{1}\mid\mid b\mid}(\sigma(b)\cdot\Delta)\circ
((a_1-(-1)^{(\mid a_{2}\mid-\mid a_{1}\mid) \mid b\mid}a_{2})\cdot\Delta)
\end{eqnarray*}
Obviously,  $a_{1}   \cdot       \Delta=a_{2}   \cdot       \Delta$ is equivalent to $(a_{1}-a_{2})\in
Ann(\Delta)$. Hence, using the assumption (\ref{HLSCond1}), we also have
$\sigma(a_{1}-a_{2})\in Ann(\Delta).$
Then, since $| a_{2}|-| a_{1}|=0$ and
$\sigma(a_{1}-a_{2})\in Ann(\Delta)$,
we obtain $$[a_{1} \cdot\Delta ,b\cdot \Delta]_{\sigma}
-[a_{2} \cdot \Delta ,b \cdot\Delta]_{\sigma}=0.$$
Similarly, we have $[b \cdot\Delta
,a_{1} \cdot\Delta]_{\sigma}=[b\cdot\Delta , a_{2}\cdot
\Delta]_{\sigma}.$

Next, we show that the superspace $\A \Delta$ is closed under
$[\cdot,\cdot]_{\sigma}$. Indeed,
\begin{eqnarray*}
&&[a \cdot\Delta,b\cdot\Delta]_{\sigma}(c)\\
&&=(\sigma(a)\cdot\Delta)\circ(b\cdot\Delta)(c)
-(-1)^{|a||b|}(\sigma (b)\cdot\Delta )\circ(a\cdot\Delta)(c)\\
&&=(\sigma(a)\cdot\Delta)(b\Delta(c))
-(-1)^{|a||b|}(\sigma (b)\cdot\Delta)(a\Delta(c))\\
&&=\sigma(a)(\Delta(b)\Delta(c)+(-1)^{|b|}\sigma (b)\Delta(\Delta(c))
-(-1)^{|a||b|}(\sigma (b)(\Delta(a)\Delta(c))
+(-1)^{|a|}\sigma (a)\Delta (\Delta(c)))
\\&& =(\sigma(a)\Delta(b)-(-1)^{|a||b|}\sigma(b)\Delta(a))\Delta(c)
+(-1)^{|b|}(\sigma(a)\sigma(b)-(-1)^{|a||b|}\sigma(b)\sigma(a))
 \Delta(\Delta(c))
\end{eqnarray*}
Since $\A$ is supercommutative the second term vanishes. Therefore, we
obtain formula (\ref{**}) which shows that the bracket is closed.

Now, we prove that formula (\ref{***}) holds. First, we consider part
given by the first term.
part.
 \begin{eqnarray*}
&& \circlearrowleft_{a,b,c}{(-1)^{|a||c|}
[\sigma(a)\cdot\Delta,[b\cdot\Delta,c\cdot\Delta]_{\sigma}]_{\sigma}}
\\&& =\circlearrowleft_{a,b,c}{(-1)^{|a||c|}
[\sigma(a)\cdot\Delta,(\sigma(b)\Delta(c)-(-1)^{|b||c|}\sigma(c)\Delta(b))\cdot\Delta]_{\sigma}}
\\ && =\circlearrowleft_{a,b,c}(-1)^{|a||c|}
(\sigma^2(a)\Delta(\sigma(b)\Delta(c)-(-1)^{|b||c|}\sigma(c)\Delta(b))
\\ && -(-1)^{|a|(|b|+|c|)}(\sigma(\sigma(b)\Delta(c)
-(-1)^{|b||c|}\sigma(c)\Delta(b))\Delta (\sigma (a)))\cdot\Delta  \\
&& =\circlearrowleft_{a,b,c}(-1)^{|a||c|}
(\sigma^2(a)(\Delta(\sigma(b))\Delta (c)+
\sigma^2(b)\Delta^2(c)
-(-1)^{|b||c|}(\Delta(\sigma(c)\Delta(b)+\sigma^2(c)\sigma^2(b)))
\\ && -(-1)^{|a|(|b|+|c|)}(\sigma^2(b)\sigma(\Delta(c))
-(-1)^{|b||c|}\sigma^2(c)\sigma(\Delta(b)))
\Delta(\sigma (a))))   \cdot      \Delta\\
&&=\circlearrowleft_{a,b,c}(-1)^{|a||c|}(\sigma^2(a)\Delta(\sigma (b))\Delta (c)
+\sigma^2(a)\sigma^2 (b))\Delta^2 (c)\\
&&
-(-1)^{|b||c|}\sigma^2(a)\Delta(\sigma(c))\Delta(b)
-(-1)^{|b||c|}\sigma^2(a)\sigma^2(c)\Delta^2(b)\\
&&
-(-1)^{|a|(|b|+|c|)}\sigma^2(b)\sigma(\Delta(c))\Delta(\sigma(a))
+(-1)^{|a||b|+|a||c|+|b||c|)}
\sigma^2(c)\sigma(\Delta(b))\Delta(\sigma(a)) )  \cdot      \Delta
 \end{eqnarray*}
 where $\sigma^2=\sigma\circ\sigma$ and $\Delta^2=\Delta\circ\Delta$.

 Applying cyclic summation to the second and fourth terms, they vanish.
 Using assumption and doing the same with the fifth and sixth in the previous
 identity we get also 0.

Finally, it remains
\begin{eqnarray}\label{part1}
&&\circlearrowleft_{a,b,c}{(-1)^{|a||c|}
[\sigma(a)\cdot\Delta,[b\cdot\Delta,c\cdot\Delta]_{\sigma}]_{\sigma}}=\\
&& \circlearrowleft_{a,b,c}(-1)^{|a||c|}(\sigma^2(a)\Delta(\sigma (b))\Delta (c)
-(-1)^{|b||c|}\sigma^2(a)\Delta(\sigma(c))\Delta(b)
 )  \cdot      \Delta \nonumber
\end{eqnarray}
Now, we consider the  part in (\ref{***})given by the second term.
\begin{eqnarray*}
&& \circlearrowleft_{a,b,c}{(-1)^{|a||c|}
\delta\ [a\cdot\Delta,[b\cdot\Delta,c\cdot\Delta]_{\sigma}]_{\sigma}}
\\&& =\circlearrowleft_{a,b,c}{(-1)^{|a||c|}
\delta\ [a\cdot\Delta,((\sigma(b)\Delta(c)-
(-1)^{|b||c|}\sigma(c)\Delta(b))\cdot\Delta]_{\sigma}}
\\ && =\circlearrowleft_{a,b,c}(-1)^{|a||c|}
\delta\ ((\sigma (a)\Delta(\sigma(b)\Delta(c)
-(-1)^{|b||c|}\sigma(c)\Delta(b))\cdot\Delta
\\ && -(-1)^{|a|(|b|+|c|)}(\sigma^2(b)(\sigma\Delta(c))
+(-1)^{|b||c|}\sigma^2(c)\sigma(\Delta(b))\Delta (a)))\cdot\Delta
\\ && =\circlearrowleft_{a,b,c}(-1)^{|a||c|}
\delta\ ((-1)^{|b||c|}\sigma (a)\Delta(\Delta(c)\sigma(b))
-\sigma (a)\Delta(\Delta(b)\sigma(c))
\\ && -(-1)^{|a|(|b|+|c|)}\sigma^2(b)(\sigma\Delta(c))
+(-1)^{|b||c|}\sigma^2(c)\sigma(\Delta(b))\Delta (a))\cdot\Delta
\\
&&=\circlearrowleft_{a,b,c}(-1)^{|a||c|}\delta\ (
(-1)^{|b||c|} \delta\ \sigma(a)\Delta^2(c)\sigma (b)
(-1)^{|b||c|}\sigma(a)\sigma(\Delta(c))\Delta(\sigma(b))
\\ &&
 -\sigma(a)\Delta^2(b)\sigma(c)
-\sigma(a)\sigma(\Delta(b))\Delta(\sigma(c))
\\ &&-(-1)^{|a|(|b|+|c|)}\sigma^2(b)\sigma(\Delta(c))\Delta(a)
+(-1)^{|b||c|}\sigma^2(c)\sigma(\Delta(b))\Delta (a))
\cdot      \Delta\\
&& =\circlearrowleft_{a,b,c}(-1)^{|a||c|} \delta\ (
(-1)^{|b||c|} \sigma(a)\Delta^2(c)\sigma(b) +
(-1)^{|b||c|}\delta\ \sigma(a)\sigma(\Delta(c))\sigma(\Delta(b))
\\ &&
 -\sigma(a)\Delta^2(b)\sigma(c)
-\delta\ \sigma(a)\sigma(\Delta(b))\sigma(\Delta(c))
\\ &&-(-1)^{|a|(|b|+|c|)}\sigma^2(b)\sigma(\Delta(c))\Delta(a)
+(-1)^{|b||c|}\sigma^2(c)\sigma(\Delta(b))\Delta (a))
\cdot      \Delta
 \end{eqnarray*}
 The four first terms vanish up to cyclic summation. It remains
 \begin{eqnarray*}
&& \circlearrowleft_{a,b,c}{(-1)^{|a||c|}
\delta\ [a\cdot\Delta,[b\cdot\Delta,c\cdot\Delta]_{\sigma}]_{\sigma}}
\\&& =\circlearrowleft_{a,b,c}{
(-(-1)^{|a||b|}\sigma^2(b)\Delta(\sigma(c))\Delta(a)
+(-1)^{|a||c|+|b||c|}\sigma^2(c)\Delta(\sigma(b))\Delta (a))
\cdot      \Delta}
\end{eqnarray*}
Which is, up to permutation, the opposite of expression (\ref{part1}). This ends
 the proof.
\end{proof}
\section{A $q$-deformed  Witt superalgebra}
In this Section, we provide an example of infinite dimensional
Hom-Lie superalgebra. We construct using the previous theorem a
realization of
 the $q$-deformed Witt superalgebra, which carries a
structure of Hom-Lie superalgebra for a suitable choice of
the twist map.

 Let $\A$ be the complex superalgebra
 $\A=\A_0 \oplus \A_1$ where $\A_0 = \mathbb{C}[t,t^{-1}]$ is  the Laurent
 polynomials in one variable and $\A_1 =\theta\ \mathbb{C}[t,t^{-1}] $,
 where $\theta$ is the Grassman variable $(\theta^2 = 0)$. We assume that $t$
 and $\theta$ commute. The generators of
 $\A$ are of the form $t^n$ and $\theta t^n$ for $n\in \mathbb{Z}$.

 Let $q\in \mathbb{C}\backslash \{0,1\}$ and $n\in\mathbb{N}$, we set
 $\{n\}=\frac{1-q^n}{1-q}$, a $q$-number. The $q$-numbers have the following
 properties $\{n+1\}=1+q \{n\}=\{n\}+q^n$ and $\{n+m\}=\{n\}+q^n\{m\}.$

  Let  $\sigma$ be
 the  algebra endomorphism on $\A$ defined by
 $$\sigma(t^n)= q^n t^n\quad \text{ and }\quad \sigma(\theta)=q \theta.$$
 Let $\partial _t$ and $\partial _\theta $ be two linear maps on $\A$
 defined by
 \begin{eqnarray*}
 \partial _t (t^n)=\{n\} t^n, & \partial _t (\theta t^n)=\{n\} \theta t^n,\\
 \partial _\theta(t^n)=0,   &\partial _\theta(\theta t^n)=q^n t^n.
 \end{eqnarray*}

 \begin{lem} The linear map  $\Delta=\partial_t +\theta \partial_\theta$
 on $\A$ is a  an even
 $\sigma$-derivation.

 Hence,
 \begin{eqnarray*}
 && \Delta (t^n)=\{n\} t^n, \\
 && \Delta (\theta t^n)=\{n+1\} \theta t^n.
 \end{eqnarray*}
 \end{lem}
\begin{proof}
We show that
\begin{eqnarray}
\Delta (t^n t^m)&=\Delta (t^n) t^m+\sigma (t^n) \Delta (t^m),\\
\Delta (t^n \theta)&=\Delta (t^n) \theta+\sigma (t^n) \Delta (\theta),\\
\ &=\Delta (\theta)t^n +\sigma (\theta) \Delta (t^n).
\end{eqnarray}
Indeed, we have $\Delta (t^n t^m)=\Delta (t^{n+m})=\{n+m\} t^{n+m}$.
On the other hand
$$\Delta (t^n) t^m+\sigma (t^n) \Delta (t^m)=
\{n\} t^{n}t^{m}+q^n t^n\{m\} t^{m}=(\{n\} +q^n
\{m\})t^{n+m}=\{n+m\} t^{n+m}.$$ Also,  we have $\Delta (\theta t^n
)=\{n+1\} \theta t^n$. On the other hand
$$\Delta (t^n) \theta+\sigma (t^n) \Delta (\theta)=
\{n\} t^{n}\theta+q^n t^n \theta=(\{n\}
+q^n )\theta t^{n}=\{n+1\} \theta t^{n},$$
and
$$\Delta (\theta)t^n +\sigma (\theta) \Delta (t^n)=
 \theta t^{n}+q  \theta \{n\} t^n=(1+q^n \{n\}
)\theta t^{n}=\{n+1\} \theta  t^{n}.$$

\end{proof}

The conditions (\ref{HLSCond1}) and (\ref{HLSCond2}) of the theorem
  (\ref{HLS}) are satisfied by the
 $\sigma$ derivation $\Delta$, with $\delta=1$. Therefore we may construct a Hom-Lie
  superalgebra on the superspace $\mathcal{V}=\A \cdot \Delta.$

  Let $\mathcal{V}=\A \cdot \Delta,$ be a superspace generated by the
  elements $X_n=t^n\cdot \Delta$
  of parity $0$
   and the elements $G_n=\theta t^n\cdot \Delta$ of parity $1$.

  Let $[-,-]_{\sigma}$ be a bracket on the
  superspace $\mathcal{V}$  defined by
  \begin{align*}
[ X_n,X_m]_\sigma &=(\{m\}-\{n\})X_{n+m} \\
 [X_n,G_m]_\sigma &= (q^n \{m+1\}-q^{m+1}\{n\}) G_{n+m}
\end{align*}
The others brackets are obtained by supersymmetry or are $0$.

Let $\alpha$ be an even linear map on $\mathcal{V}$ defined
on the generators by
\begin{align*}
\alpha ( X_n ) &=(1+ q^n)X_{n} \\
 \alpha ( G_n ) &=(1+ q^{n+1})G_{n} \\
\end{align*}

\begin{prop}
The triple  $(\mathcal{V},[-,-]_{\sigma},\alpha)$
   is a Hom-Lie
  superalgebra.
\end{prop}
The supersymmetry follow from the definition and the Hom-superJacobi identity
follows from the theorem (\ref{HLS}). One can also check directly that
$$ \circlearrowleft_{X_n,X_m,G_p}{[\alpha (X_n),[X_m,G_p]_\sigma]_\sigma}=0.$$

The algebra constructed is a realization of the $q$-deformed Witt algebra.

\paragraph{\textbf{Acknowledgements.}  This work was supported by Mulhouse and Sfax universities
and partially  by the  Swedish Links program of SIDA Foundation. }

\end{document}